\documentclass[12pt, reqno]{amsart}
\usepackage{amsmath, amsthm, amscd, amsfonts, amssymb, graphicx, color}
\usepackage{setspace}
\usepackage{mathrsfs}
\usepackage{multicol}
\usepackage[bookmarksnumbered, colorlinks, plainpages]{hyperref}
\hypersetup{colorlinks=true,linkcolor=red, anchorcolor=green, citecolor=cyan, urlcolor=red, filecolor=magenta, pdftoolbar=true}

\newtheorem{theorem}{Theorem}[section]
\newtheorem{lemma}[theorem]{Lemma}
\newtheorem{proposition}[theorem]{Proposition}
\newtheorem{corollary}[theorem]{Corollary}
\theoremstyle{definition}

\theoremstyle{remark}

\numberwithin{equation}{section}

\newcommand{\NN}{\mathbb{N}}

\newcommand{\CC}{\mathbb {C}}

\begin{document}
\setcounter{page}{1}
\title[Path connected components of Volterra-type integral operators  ]{ Path connected components of the space of Volterra-type integral operators}
\author[Tesfa  Mengestie]{Tesfa  Mengestie }
\address{Department of Mathematical Sciences \\
Western Norway University of Applied Sciences\\
Klingenbergvegen 8, N-5414 Stord, Norway}
\email{Tesfa.Mengestie@hvl.no}
 \begin{abstract}
We study  the topological structure of the  space of Volterra-type integral operators on Fock spaces  endowed with  the operator norm. We proved that the space  has the same connected  and path connected components which is  the set of all compact operators acting  on the  Fock spaces. We also obtained a characterization of isolated points of the space of the operators and showed that there exists  no essentially isolated Volterra-type integral operator.
\end{abstract}
\maketitle
\section{introduction}
The theory of  Volterra-type integral operators have been a subject of  high interest during the  past   three  decades. The operators have been studied quite extensively on various functional spaces  over several domains. Most of the studies made aimed at  describing  their operator-theoretic properties  in terms of the function-theoretic properties of  the symbols inducing them. See for example \cite{Alsi1, Alsi2, Olivia, Olivia1, TM, TM0, JPP,JPP1, Jord}
and the related  references therein. In contrast, there have not been much  effort to understand the topoplogical structure of the space of the  operators. Recentely, the compact difference structure of
the operators have been described in \cite{TMMW}. In  this note,  we continue that line of research and describe the path connected components and the connected components  of  the space of the operators   acting on Fock spaces.

 We recall that if $g$ is a holomorphic function over a given domain, the Volterra-type integral operator $V_g$ is defined by
\begin{align*}
V_gf(z)= \int_{0}^z f(w)g'(w) dw.
\end{align*}
Let $\CC$ be the complex plane and  $0<p\leq \infty $.  Then   the classical
Fock spaces $\mathcal{F}_p$ consist of entire functions $f$ for which
\begin{align*}
\|f\|_{p}^p=  \frac{p}{2\pi}\int_{\CC} |f(z)|^p
e^{-\frac{p }{2}|z|^2} dA(z) <\infty,\ \  0<p<\infty \ \ \text{and}\\
\|f\|_\infty= \sup_{z\in \CC} |f(z)|e^{-\frac{1}{2}|z|^2} <\infty, \ \ p= \infty,\quad  \quad \quad \quad \quad  \quad \quad \quad
\end{align*}
where $dA$  denotes the Lebesgue area  measure.
 In particular,  the space  $\mathcal{F}_2$ is a
reproducing kernel Hilbert space with kernel and normalized reproducing kernel
functions   given by the  explicit formulas
\begin{align*} K_{w}(z)= e^{ \overline{w}z}\ \   \text{and} \ \
k_{w}(z)= e^{\overline{w}z-\frac{|w|^2}{2}}.
\end{align*}
 Furthermore, a simple computation shows that the kernel function $K_{w}$ belongs to all the
  Fock space $\mathcal{F}_p$  with  norms
  \begin{align}
  \label{kernelnorm}
  \|K_{w}\|_p= e^{\frac{1}{2}|w|^2}
  \end{align} for all  $w\in \CC$ and $0<p\leq \infty$. \\
   A crucial ingredient to prove our results is the Littlewood-Paley type estimate for Fock spaces. For $0<p<\infty,$ the estimate  was proved in \cite{Olivia} and reads \footnote{The  notation $U(z)\lesssim V(z)$ (or
equivalently $V(z)\gtrsim U(z)$) means that there is a constant
$C$ such that $U(z)\leq CV(z)$ holds for all $z$ in the set of a
question. We write $U(z)\simeq V(z)$ if both $U(z)\lesssim V(z)$
and $V(z)\lesssim U(z)$.}
  \begin{align}
 \label{Olivia}
 \int_{\CC} |f(z)|^p e^{-\frac{p}{2}|z|^2} dA(z) \simeq |f(0)|^p + \int_{\CC} \frac{|f'(z)|^p}{(1+|z|)^p} e^{-\frac{p}{2}|z|^2} dA(z)
 \end{align} for all $f$ in $\mathcal{F}_p$. The analogous result for  $p= \infty$ follows from \cite{TM1} and becomes
 \begin{align}
 \label{Paley}
\|f\|_\infty \simeq |f(0)|+ \sup_{z\in\CC} \frac{|f'(z)|}{1+|z|} e^{-\frac{1}{2}|z|^2}.
 \end{align}
   We close this section with a word on notation.  We denote by   $\textbf{V}(\mathcal{F}_p, \mathcal{F}_q)$   the space of all bounded  Volterra-type integral operators $V_g : \mathcal{F}_p \to \mathcal{F}_q$ equipped with the operator norm topology unless otherwise specified.\\
    \section{The main results}
    We start this section stating  our first main result on path connected components of the space $\textbf{V}(\mathcal{F}_p, \mathcal{F}_q)$.
\begin{theorem}
  \label{thm1}
   Let $0<p, q \leq \infty$ and $V_g : \mathcal{F}_p \to \mathcal{F}_q$ be a compact operator. Then $V_g$ and $V_{g(0)}$ belong to the same path connected components of
   the space $\textbf{V}(\mathcal{F}_p, \mathcal{F}_q)$.
\end{theorem}
\begin{proof} Since $V_g$ is compact, by Corollary~2 of \cite{TM} and Theorem~1.1 of  \cite{TM1}, the symbol $g$ has  affine form  $g(z)= az+b$. If $a= 0,$  then  $V_g= V_{g(0)}= 0$ and the assertion of the result  holds trivially.
Thus,  we assume that $a\neq0$ and consider  a sequence of scaling functions $g_t:[0,1] \to \CC ,  \ \  g_t(z)= g(tz)$.  Then,  $V_{g_t}: \mathcal{F}_p \to \mathcal{F}_q$  is compact for all  $t$ and satisfies
\begin{align*}
V_g= V_{g_1} \quad \text{and} \ \ V_{g(0)}= V_{g_0}.
\end{align*}
We define an operator $T: [0,1] \to \textbf{V}(\mathcal{F}_p, \mathcal{F}_q)$ by
$T(x)= V_{g_x}$. Then, to prove our results  it suffices to show that  for every $x$ in $[0,1]$
\begin{align*}
\lim_{t\to x} \| V_{g_t}-V_{g_x}\|_q =0.
\end{align*}
For the case when $0<p, q <\infty$,  applying the Littlewood--Paley type estimate \eqref{Olivia}  and linearity of the integral we have
\begin{align}
\label{new}
\| V_{g_t}f-V_{g_x}f\|_q^q \simeq \int_{\CC} \frac{|g'_t(z)-g'_x(z)f(z)|^q}{(1+|z|)^q} e^{-\frac{q}{2}|z|^2}dA(z)\quad \quad \quad \quad \quad \quad \quad \quad\nonumber\\
= |at-ax|^q\int_{\CC} \frac{|f(z)|^q}{(1+|z|)^q} e^{-\frac{q}{2}|z|^2}dA(z)\quad \quad \quad \quad \quad \quad \quad\nonumber\\
= |a|^q  |t-x|^q\int_{\CC}\frac{|f(z)|^q}{(1+|z|)^q} e^{-\frac{q}{2}|z|^2}dA(z).\end{align}
We now consider the case when $p\leq q<\infty$. For this,  we observe that
\begin{align}
\label{iint}
\int_{\CC}\frac{|f(z)|^q}{(1+|z|)^q} e^{-\frac{q}{2}|z|^2}dA(z) \leq\int_{\CC} |f(z)|^q e^{-\frac{q}{2}|z|^2}dA(z)\simeq \|f\|_q^q \leq \|f\|_p^q,
\end{align} where the last inequality follows from the inclusion $\mathcal{F}_p\subseteq\mathcal{F}_q$ whenever $p\leq q$.  \\
On the other hand, if $p\leq q= \infty$, then  applying \eqref{Paley}
\begin{align}
\label{neww}
\| V_{g_t}f-V_{g_x}f\|_\infty \simeq \sup_{z\in\CC} \frac{|(g'_t(z)-g'_x(z))f(z)|}{1+|z|} e^{-\frac{1}{2}|z|^2} \quad \quad \quad \quad \quad \nonumber\\
= |at-ax| \sup_{z\in\CC} \frac{|f(z)|}{1+|z|} e^{-\frac{1}{2}|z|^2}
\leq |at-ax| \sup_{z\in\CC} |f(z)| e^{-\frac{1}{2}|z|^2} \nonumber\\
\leq |a||t-x| \|f\||_\infty \leq|a||t-x| \|f\||_p.\end{align}
If   $ q<p<\infty$, then  applying  H\"older's inequality, the integral in  \eqref{new} is further  estimated by
\begin{align}
\label{int}
 \int_{\CC}\frac{|f(z)|^q}{(1+|z|)^q} e^{-\frac{q}{2}|z|^2}dA(z)\leq \bigg(\int_{\CC}|f(z)|^p  e^{-\frac{p}{2}|z|^2}dA(z) \bigg)^{\frac{q}{p}} \quad \quad  \quad \quad   \quad \quad\nonumber\\
 \ \quad \quad \quad \quad \times \bigg(\int_{\CC} \frac{dA(z)}{(1+|z|)^\frac{qp}{p-q}}\bigg)^{\frac{q-p}{p}}
 \simeq \|f\|_p^q  \bigg(\int_{\CC} \frac{dA(z)}{(1+|z|)^\frac{qp}{p-q}}\bigg)^{\frac{q-p}{p}}.
   \end{align}
 Since $V_g$ is compact, by Theorem~1.1 of \cite{TM000}, it holds that  $q> 2p/(p+2)$ and hence the last  integral in \eqref{int} is finite.  \\
It remains to  show when $p= \infty$  and $ q <\infty$. To this end, using \eqref{new},
\begin{align}
\label{finite}
\| V_{g_t}f-V_{g_x}f\|_q^q \simeq  |a|^q  |t-x|^q\int_{\CC}\frac{|f(z)|^q}{(1+|z|)^q} e^{-\frac{q}{2}|z|^2}dA(z)\quad \quad \quad \quad \nonumber\\
\leq |a|^q  |t-x|^q \|f\|_\infty^q \int_{\CC}\frac{dA(z)}{(1+|z|)^q} \lesssim  |a|^q  |t-x|^q \|f\|_\infty^q,
\end{align}  where the last estimate is possible since $V_g: \mathcal{F}_\infty \to \mathcal{F}_q$ is bounded, by Theorem~1.1 of  \cite{TM1}, $q$ has to be bigger than 2, and hence the last integral above converges.\\
From  the series of estimates in \eqref{new}, \eqref{iint}, \eqref{neww}, \eqref{int},  and \eqref{finite} we arrive at
\begin{align*}
\| V_{g_t}-V_{g_x}\|_q \lesssim  |t-x|\|f\|_p
\end{align*}  for all $0<p,q \leq \infty$ from which we deduce that  the  assertion of the theorem holds.\\
\end{proof}
An immediate consequence of the above  main result is that the set of compact Volterra-type integral operators acting between Fock spaces is path  connected. Indeed, if  $V_{g_1}$ and $V_{g_2}$ are two compact Volterra-type integral operators, then
   $g_1(z)= az+b$ and $g_2(z)= cz+d$. By Theorem~\ref{thm1}, it  also follows that $V_{g_1}$ and $V_{g_1(0)}$ belong to the same path connected component. The same holds true for $V_{g_2}$ and $V_{g_2(0)}$. But $V_{g_1(0)}=V_{g_2(0)}$  is the zero operator and hence the assertion.  We record this observation as a corollary below for the sake of easier further referencing.
\begin{corollary}\label{thm2}
 Let $0<p, q \leq \infty$. Then the set of all compact operators  $V_g : \mathcal{F}_p \to \mathcal{F}_q$ is a path connected component of
   the space $\textbf{V}(\mathcal{F}_p, \mathcal{F}_q)$.
  \end{corollary}
 If   $0<q<p\leq \infty$, it is shown in \cite{TM000} and \cite{TM1} that the operator $  V_g: \mathcal{F}_p\to  \mathcal{F}_q$ is bounded if and only if it is  compact.
An immediate consequence of this fact  and Theorem~\ref{thm2} above  is that  the whole space $\textbf{V}(\mathcal{F}_p, \mathcal{F}_q)$ is path connected  when  $0<q<p\leq\infty$. It also means that there exists no isolated (singleton component)  bounded Volterra-type integral  operator  in this case.  For the case when $0<p\leq q <\infty$,  our next main result characterizes  the isolated Volterra-type integral operators  in $V_g(\mathcal{F}_p, \mathcal{F}_q)$.
\begin{theorem}\label{thm3}
Let $0<p\leq q \leq\infty$ and $V_g : \mathcal{F}_p \to \mathcal{F}_q$ be a bounded operator. Then the following statements are equivalent:
\begin{enumerate}
\item $V_g$ is isolated in $\textbf{V}(\mathcal{F}_p, \mathcal{F}_q)$;
\item $V_g$ is not  compact, and hence $g(z)= az^2+bz+c$ with $a\neq0$;
\end{enumerate}
\end{theorem}
\begin{proof} The assertion (i) implies (ii) follows easily by an application of Corollary~\ref{thm2}. On the other hand, if (ii) holds, by Theorem~2 of \cite{TM},  $g$ has the form  $g(z)= az^2+bz+c$ and $a\neq0$. Then  to conclude (i), it suffices to show that there exists a positive number $m$ such that \begin{align*}\|V_g -V_{g_1}\|\geq m\end{align*}  for all complex  polynomials of the form
$g_1(z)= dz^2+ez+f$  with   either $ a\neq d $ or $ b\neq e $. In view of this, when  $p\leq q <\infty$, we have
\begin{align*}
\|V_g-V_{g_1}\|^q \geq
\|V_gk_w-V_{g_1}k_w\|_q^q  \geq C \int_{\CC} \frac{|g'(z)-g_1'(z)|^q |k_w(z)|^q}{(1+|z|)^q} e^{-\frac{q}{2}|z|^2} dA(z)\\
= C \int_{\CC} \frac{|2z(a-d)+ b-e|^q |k_w(z)|^q}{(1+|z|)^q} e^{-\frac{q}{2}|z|^2} dA(z)
\end{align*}
We first assume that $a\neq d$ and   further estimate the last integral from below as
\begin{align*}
C \int_{\CC} \frac{|2z(a-d)+ b-e|^q |k_w(z)|^q}{(1+|z|)^q} e^{-\frac{q}{2}|z|^2} dA(z)\geq C_1 \int_{\CC} |k_w(z)|^q e^{-\frac{q}{2}|z|^2} dA(z)\\
\geq C_1 \int_{D(w,1)} |k_w(z)|^q e^{-\frac{q}{2}|z|^2} dA(z),
\end{align*}  for all $w\in \CC$ where $D(w,1)$ is a disc of center $w$ and radius 1 in the complex plane,  and $C$ and $C_1$ are some positive constants. Since $|k_w|^q$ is subharmonic, it follows that
\begin{align*}
  C_1 \int_{D(w,1)} |k_w(z)|^q e^{-\frac{q}{2}|z|^2} dA(z)\geq C_2 |k_w(w)|^q e^{-\frac{q}{2}|w|^2}= C_2>0,
\end{align*}
where we use \eqref{kernelnorm} for the last equality.  Then, we may
  set  $m= C_2$ to claim the assertion in this case.

If $a= d$, then by hypothesis,  $b\neq e$.  Furthermore, for each $z\in D(w,1)$, it holds that $1+|z| \simeq 1+|w|$. Using this and arguing as above we have
\begin{align*}
C \int_{\CC} \frac{|b-e|^q |k_w(z)|^q}{(1+|z|)^q} e^{-\frac{q}{2}|z|^2} dA(z)\geq \frac{C_3}{(1+|w|)^q} \int_{D(w,1)} |k_w(z)|^q e^{-\frac{q}{2}|z|^2} dA(z)\\
\geq  \frac{C_4}{(1+|w|)^q} |k_w(w)|^q e^{-\frac{q}{2}|w|^2}= \frac{C_4}{(1+|w|)^q}
\end{align*} for all $w\in \CC$. In particular setting $w=0$ and hence $C_4= m$,  we arrive at  the desired  assertion.

For the case when $p\leq q= \infty$, applying \eqref{Paley} and arguing as above we obtain
\begin{align*}
\|V_g-V_{g_1}\| \geq
\|V_gk_w-V_{g_1}k_w\|_\infty \geq C \sup_{z\in \CC} \frac{|g'(z)-g_1'(z)| |k_w(z)|}{1+|z|} e^{-\frac{1}{2}|z|^2}\\
\geq C \frac{|g'(z)-g_1'(z)| |k_w(z)|}{1+|z|} e^{-\frac{1}{2}|z|^2}
\geq C  \frac{|2w(a-d)+ b-e| }{1+|w|}
\end{align*}  for all $w\in \CC$. In particular,  setting $w=0$ and  $ m= C  | b-e| $ we arrive at our conclusion whenever $b\neq e$. If $b= e$, then $a\neq d$ and  letting $|w|\to \infty$, we observe that
$
\|V_g-V_{g_1}\| \geq  2C  |a-d|= m.
$
  \end{proof}
Interesting consequence of this result and Corollary~\ref{thm2} is that
 the space  $V_g(\mathcal{F}_p, \mathcal{F}_q)$ has the same connected  and path connected components which is just the set of all compact operators  acting between the Fock spaces.
\subsection{Essentially connected }
A natural question  following Theorem~\ref{thm3} is whether every isolated Volterra-type integral operator in  $\textbf{V}(\mathcal{F}_p, \mathcal{F}_q)$
is  still isolated under the essential norm topology which is weaker than the topology induced by the operator norm.  The main result of this section shows   this is not the case. In deed, we will show that there exists no  essentially isolated point in the space  $\textbf{V}(\mathcal{F}_p, \mathcal{F}_q)$ when endowed with the essential norm.

 We recall that  for Banach spaces $\mathscr{H}_1$ and $\mathscr{H}_2$, the
 essential norm
$\|T\|_e$ of a bounded linear operator $T:\mathscr{H}_1  \to  \mathscr{H}_2 $
is defined by
\begin{align*}
\|T\|_e= \inf_{K} \big\{\|T-K\|; \ \ K: \mathscr{H}_1 \to \mathscr{H}_2\text{ is a compact operator }\big\}.
\end{align*}
In particular, the operator  $T$ is compact if and only if its essential norm is zero.

 Let us first prove the following proposition which is interest of its own, and  provides a useful and interesting estimate for the essential norm of the Volterra-type integral operators on Fock spaces.
 \begin{proposition}\label{prop1}
 Let $1\leq p\leq q\leq \infty$ and  $V_g: \mathcal{F}_p \to \mathcal{F}_q $ be a bounded operator. That is
   $g(z)= az^2+bz+c$. Then
 \begin{align}
 \label{ess}
 \| V_g\|_e \simeq |a|.
 \end{align}
  \end{proposition}
 \begin{proof}
 If $q<\infty,$ then by a particular case of Theorem~3 in \cite{TM},  we have
 \begin{align}
 \label{ess0}
 \| V_g\|_e^q \simeq \limsup_{|w|\to \infty}\int_{\CC} \frac{|g'(z)|^q |k_w(z)|^q}{(1+|z|)^q} e^{-\frac{q}{2}|z|^2} dA(z).
 \end{align}
 We proceed to estimate  the integral in \eqref{ess0} from both above and below to arrive at \eqref{ess}.  First observe that
  \begin{align*}
 \int_{\CC} \frac{|g'(z)|^q |k_w(z)|^q}{(1+|z|)^q} e^{-\frac{q}{2}|z|^2} dA(z) \geq \int_{D(w,1)} \frac{|g'(z)|^q |k_w(z)|^q}{(1+|z|)^q} e^{-\frac{q}{2}|z|^2} dA(z)\quad \quad \quad \nonumber\\
 \simeq \frac{1}{(1+|w|)^q}\int_{D(w,1)} |g'(z)|^q |k_w(z)|^q e^{-\frac{q}{2}|z|^2} dA(z)\\
 \gtrsim \frac{1}{(1+|w|)^q} |g'(w)|^q |k_w(w)|^q e^{-\frac{q}{2}|w|^2}
 = \frac{|g'(w)|^q}{(1+|w|)^q}  = \frac{|2aw+b|^q}{(1+|w|)^q}.
 \end{align*}
 It follows from this and \eqref{ess0} that
 \begin{align}
  \| V_g\|_e^q\gtrsim \limsup_{|w|\to \infty} \frac{|2aw+b|^q}{(1+|w|)^q} \simeq |a|^q
 \end{align} and one side of the inequality in \eqref{ess} follows. To prove the other side, using the inequality $|a+b|^q \leq 2^q(|a|^q+ |b|^q)$ and \eqref{kernelnorm},  we estimate the integral  in \eqref{ess0} also  have
 \begin{align*}
 \int_{\CC} \frac{|g'(z)|^q |k_w(z)|^q}{(1+|z|)^q} e^{-\frac{q}{2}|z|^2} dA(z) \lesssim  |a|^q\int_{\CC} |k_w(z)|^q e^{-\frac{q}{2}|z|^2} dA(z)\nonumber\\
+ |b|^q\int_{\CC}\frac{|k_w(z)|^q}{(1+|z|)^q}  e^{-\frac{q}{2}|z|^2} dA(z) \simeq |a|^q+ \frac{|b|^q}{(1+|w|)^q}
   \end{align*} from which  we deduce the other side inequality
  \begin{align}
  \| V_g\|_e^q\lesssim \limsup_{|w|\to \infty} |a|^q+ \limsup_{|w|\to \infty}\frac{|b|^q}{(1+|w|)^q}=  |a|^q.
  \end{align}
  On the other hand, if $q= \infty$, using the fact that  $k_w$ is a weakly null sequence and eventually  \eqref{kernelnorm}, we estimate  the essential norm from below  by
  \begin{align*}
    \| V_g\|_e \geq  \limsup_{|w|\to \infty} \|V_gk_w\|_\infty \simeq \limsup_{|w|\to \infty} \  \sup_{z\in \CC} \frac{|g'(z)|k_w(z)|e^{-\frac{1}{2}|z|^2}}{1+|z|}\quad \quad \nonumber\\
  \geq \limsup_{|w|\to \infty}  \frac{|g'(z)|k_w(z)|e^{-\frac{1}{2}|z|^2}}{1+|z|}
   \geq \limsup_{|w|\to \infty}  \frac{|g'(w)|k_w(w)|e^{-\frac{1}{2}|w|^2}}{1+|w|}\\
   = \limsup_{|w|\to \infty}  \frac{|aw+b| }{1+|w|}\simeq |a|.
     \end{align*}
     To prove the upper estimates in \eqref{ess}, we consider  maps $\psi_k:\CC \to \CC,\ \ \psi_k(z)=\frac{k}{k+1} z$ for each $k\in \NN$.  Then  $C_{\psi_k}: \mathcal{F}_p \to \mathcal{F}_\infty$ constitutes  a sequence of compact composition operators  for all $p\geq 1$. On the other hand, if  $V_g$ is bounded, then $V_g( C_{\psi_k}):\mathcal{F}_p \to \mathcal{F}_\infty$ also  constitutes a sequence of compact operators.  Using this,
\begin{align*}
\|V_g\|_e \leq \|V_g-V_g( C_{\psi_k})\|= \sup_{\|f\|_p\leq 1} \|(V_g-V_g(C_{\psi_k}))f\|_\infty\quad \quad \quad\nonumber\\
\simeq \sup_{\|f\|_p\leq 1}\  \sup_{z\in \CC}\frac{|g'(z)|\Big|f(z)-f(\psi_k(z))\Big|}{1+|z|}e^{-\frac{1}{2}|z|^2}.
\end{align*}
The supremum above is comparable to the quantity
\begin{align}
\label{sum0}
\sup_{\|f\|_p\leq 1}\   \sup_{|z|>r}\frac{|g'(z)|}{1+|z|}\Big|f(z)-f(\psi_k(z))\Big|e^{-\frac{1}{2}|z|^2}\quad \quad \quad \quad \nonumber\\
+\sup_{\|f\|_p\leq 1} \ \sup_{|z|\leq r}\frac{|g'(z)|}{1+|z|}\Big|f(z)-f(\psi_k(z))\Big|e^{-\frac{1}{2}|z|^2}
\end{align} for a certain fixed positive number $r$. Then,  the first  summand  above is bounded  by
\begin{align}
\label{partly0}
 \sup_{|z|>r} \bigg(\frac{|g'(z)|}{1+\psi'(z)}\bigg) \sup_{\|f\|_p\leq 1}  \sup_{|z|>r}\bigg( \big|f(z)-f(\psi_k(z))\big|e^{-\frac{1}{2}|z|^2}\bigg)\nonumber\\
\leq  \sup_{|z|>r} \bigg(\frac{|g'(z)|}{1+|z|}\bigg) \sup_{\|f\|_p\leq 1} \|f\|_{\infty}\leq  \sup_{|z|>r} \bigg(\frac{|g'(z)|}{1+|z|}\bigg)\sup_{\|f\|_p\leq 1} \|f\|_{p}\nonumber\\
\leq \sup_{|z|>r} \frac{|az+b|}{1+|z|}.
\end{align}
As for the second summand in \eqref{sum0}, we observe  that by integrating the function $f'$ along
 the radial segment $ [\frac{kz}{k+1}z, z]$   we find
\begin{align*}
\bigg|f(z)-f\Big(\frac{k}{k+1}z\Big)\bigg| \leq \frac{|z||f'(z^*)|}{k+1}
\end{align*}  for some $z^*$ in  the radial segment $ [\frac{kz}{k+1}z, z]$. By Cauchy estimate's for $f'$,
\begin{align*}
|f'(z^*)| \leq \frac{1}{r} \max_{|z|=2r} |f(z)|,
\end{align*}
and hence
\begin{align*}
\bigg|f(z)-f\Big(\frac{k}{k+1}z\Big)\bigg| \leq \frac{|z|}{r(k+1)} \max_{|z|=2r} |f(z)|.
\end{align*}
The above estimates  ensure that
\begin{align*}
\frac{|g'(z)|}{1+|z|} \Big|f(z)-f\bigg(\frac{k}{k+1}z\bigg)\Big| e^{-\frac{1}{2}|z|^2} \leq \frac{|z|}{r(k+1)} \sup_{z\in \CC}\bigg(\frac{|az+b|}{1+|z|}e^{-\frac{1}{2}|z|^2}\bigg) \max_{|z|=2r} |f(z)|\nonumber\\
\lesssim \frac{|z|}{r(k+1)} \max_{|z|=2r} |f(z)|.
\end{align*}
Using the fact that $|f(z)|e^{-|z|^2/2} \leq \|f\|_\infty$   we further estimate
\begin{align*}
\max_{|z|=2r} |f(z)|\leq  \max_{|z|=2r} e^{\frac{1}{2}|z|^2} \|f\|_\infty \leq e^{2r^2} \|f\|_p.
\end{align*}
Now combining all the above estimates, we find that the second piece of the sum in \eqref{sum0} is bounded by
\begin{align*}
\sup_{\|f\|_p\leq 1} \ \sup_{|z|\leq r}\frac{|g'(z)|}{1+|z|}\big|f(z)-f(\psi_k(z))\big|e^{-\frac{1}{2}|z|^2}
\lesssim \frac{ e^{2r^2}}{k+1}   \rightarrow 0 \ \text{as} \ k \rightarrow \infty,
\end{align*}
 from which, \eqref{partly0},  and since $r$ is arbitrary, we  deduce
\begin{align*}
\|V_g\|_{e}\lesssim \sup_{|z|>r} \frac{|g'(z)|}{1+|z|}= \limsup_{|z| \to \infty} \frac{|az+b|}{1+|z|}\simeq |a|.
\end{align*}
   \end{proof}
   We may now state the main result of this section on essentially isolated points.
\begin{theorem} \label{thm4}
 Let $1\leq p\leq q\leq \infty$. Then there exists no essentially isolated Volterra-type integral operator in the space  $\textbf{V}(\mathcal{F}_p, \mathcal{F}_q)$.
   \end{theorem}
 \begin{proof}
 By Corollary~\ref{thm2}, all compact Volterra-type integral operators are connected. This together with the fact that
 the essential norm topology is weaker than the operator norm topology, no compact operator is essentially isolated. Thus, we consider
     an operator $V_g\in  \textbf{V}(\mathcal{F}_p, \mathcal{F}_q)$, and assume that it is   isolated in the operator norm topology. Then by Theorem~\ref{thm3}, $V_g$ is not compact and hence $g(z)= az^2+bz+c$  with  $a\neq 0$. We plan to show that $V_g$ is not essentially isolated. That is every open ball with positive radius and center  $V_g$ is not contained in the singleton set $\{ V_g\}$. If  $V_{g_1}$ with $g_1(z)= a_1z^2+b_1z+c_1$  belongs to $  \textbf{V}(\mathcal{F}_p, \mathcal{F}_q)$,  then by  applying linearity of the integral and Proposition~\ref{prop1},
 \begin{align*}
 \| V_{g_1}-V_{g_1}\|_e = \| V_{g-g_1}\|_e \simeq |a-a_1|
 \end{align*}  which gives an estimate for the essential norm of the difference of  two Volterra-type integral operators.
 Now choosing $g_1(z)= a_1z^2+b_1z+c_1$ in such a way that $a=a_1$ and $V_{g_1}\neq V_g$, we  observe that $g_1$ belongs to every ball of center $V_g$ and positive radius. But $V_{g_1}$ does not belong to $\{ V_g\}$ and completes the proof.
   \end{proof}
For each $a\in \CC$, we set
    \begin{align*}
    V_a= \Big\{V_{g_a}\in \textbf{V}(\mathcal{F}_p, \mathcal{F}_q): g_a(z)= az^2+bz+c,  \ \ b, c\in \CC \Big \}.
    \end{align*}
    Then if $ a=0,  V_0$  corresponds to the set of compact operators in $\textbf{V}(\mathcal{F}_p, \mathcal{F}_q)$ which is path connected by Corollary~\ref{thm2}.
    If  $ a\neq 0$,  then $V_{g_a}$ is not essentially isolated by Theorem~\ref{thm4}.
    As a consequence of all these, we get the following.
    \begin{corollary}
    \label{cor1}
    Let $1\leq p\leq q\leq \infty$. Then  $ \textbf{V}(\mathcal{F}_p, \mathcal{F}_q)$ has the following essentially path connected components:
    \begin{align*}
    \textbf{V}(\mathcal{F}_p, \mathcal{F}_q)= \bigcup_{a\in \CC} V_a.
    \end{align*}
    \end{corollary}


\begin{thebibliography}{BRSHZE}
\bibitem{Alsi1} A. Aleman  and A. Siskakis, An integral operator on $H^p$, {\it Complex Variables} \textbf{28} (1995), 149--158.

\bibitem{Alsi2} A. Aleman  and A. Siskakis, Integration operators on Bergman spaces, {\it Indiana University Math J.} \textbf{46} (1997), 337--356.
\bibitem{Olivia} O. Constantin, Volterra type integration operators on Fock spaces, {\it Proc. Amer. Math. Soc.} \textbf{140} (12) (2012), 4247--4257.

    \bibitem{Olivia1} O. Constantin and Jos\'{e} \'{A}ngel Pel\'{a}ez, Integral Operators, Embedding Theorems and a Littlewood--Paley Formula on Weighted Fock Spaces,  \emph{J. Geom. Anal}.,  (2015) 1--46.

\bibitem{jmaa345} S.~Li and S.~Stevi\' c, Products of Volterra type
operator and composition operator from $H^\infty$ and Bloch spaces
to the Zygmund space, {\it J. Math. Anal. Appl.} {\bf 345} (2008),
40--52.

\bibitem{TM0} T. Mengestie, Generalized Volterra Companion Operators
on Fock Spaces, \emph{Potential Anal}, \textbf{44 }(2016), 579--599.
\bibitem{TM} T. Mengestie, Product of Volterra type integral and composition operators on weighted Fock spaces,
\emph{ J. Geom. Anal.}, \textbf{24}(2014), 740--755.

\bibitem{TM1} T. Mengestie, On the Spectrum of Volterra-Type Integral Operators
on Fock--Sobolev Spaces,\emph{ Complex Anal. Oper. Theory}, \textbf{11} (2017) 1451--1461.

\bibitem{TMMW} T. Mengestie and Mafuz Worku, Topological structures of  Volterra-type integral  operators, to appear in Mediterranean Journal of Mathematics,2017 .
\bibitem{TM000} T. Mengestie, Spectral properteis of Volterra-type integral operatos on Fock--Sobolev spaces, \emph{J. Korean Math. Soc.}  \textbf{54}
(2017),  6,  1801--1816.

    \bibitem{JPP} J. Pau and J. A. Pel\'{a}ez,  Embedding theorems and  integration operators on Bergman spaces  with rapidly decreasing weightes. \emph{J. Funct. Anal.},  \textbf{259 (10)}(2010), 2727--2756.

    \bibitem{JPP1} J. Pau and J. A. Pel\'{a}ez, Volterra type operators on Bergman spaces with exponential weights, \emph{ Contemporary Mathematics},  \textbf{561}(2012), 239--252.

    \bibitem{Jord} J. Pau, Integration operators between Hardy spaces of the unit ball of $\CC^n$, \emph{J. Funct. Anal.}, \textbf{270} (2016), 134--176.


\bibitem{KZH1} K. Zhu, Operator Theory in Function Spaces, Marcel Dekker Ink, New
York, 1990.
\end{thebibliography}
\end{document}